\documentclass[12pt,hidelinks, reqno]{amsart}
\usepackage{tikz,dsfont}
\usepackage{graphicx}
\usepackage{amsmath, amssymb, graphicx, mathrsfs, hyperref}
\usepackage{verbatim} 
\usepackage{tabulary}
\usepackage{enumerate} 
\newcommand{\NN}{\mathbb{N}}

\newtheoremstyle{dotless}{}{}{\itshape}{}{\bfseries}{}{ }{} 
\theoremstyle{dotless}

\newtheorem{theorem}{Theorem}[section]
\newtheorem{lemma}[theorem]{Lemma}

\newtheorem{prop}[theorem]{Proposition}

\newtheorem{ques}[theorem]{Question}
\newtheorem{corollary}[theorem]{Corollary}

\newtheorem{example}[theorem]{Example}

\begin{document}

\pagenumbering{arabic} \setcounter{page}{1}

\title{On distinct consecutive $r$-differences.}

\author[J. Li]{Junxian Li}\
\address
{University of Illinois Urbana-Champaign,
	Urbana, Illinois 61801}\email{jli135@illinois.edu}
\author[G. Shakan]{George Shakan}
\address
{University of Illinois Urbana-Champaign,
	Urbana, Illinois 61801}\email{george.shakan@gmail.com}

\begin{abstract}
	Suppose $A\subset \mathbb{R}$ of size $k$ has distinct consecutive $r$--differences, that is for $1 \leq i \leq k -r$, the $r$--tuples $$(a_{i+1} - a_i , \ldots , a_{i+r} - a_{i + r -1})$$ are distinct. Then for any finite 
	$B \subset \mathbb{R}$, one has $$|A+B| \gg_r |A||B|^{1/(r+1)}.$$ Utilizing de Bruijn sequences, we show this inequality is sharp up to the constant. 
	
	Moreover, for the sequence $\{n\alpha\}$, a sharp upper bound for the size of the distinct consecutive $r$--differences is obtained, which generalizes Steinhaus' three gap theorem. A dual problem on the consecutive $r$--differences of the returning times for some $\phi \in \mathbb{R}$ defined by $\{T : \{T\theta\}<\phi\}$ is also considered, which generalizes a result of Slater.
\end{abstract}

\maketitle



\section{Introduction}

Given $A,B \subset \mathbb{R}$ finite, we define the {\em sumset} $$A+B = \{a+b : a \in A , b \in B\}.$$  Let $A = \{a_1 < \ldots < a_k\}.$ We say $A$ is {\em convex} if for all $1 < i < k$$$a_i - a_{i-1} < a_{i+1} - a_i.$$ Hegyv\'{a}ri \cite{He}, answering a question of Erd\"{o}s, proved that if $A$ is convex then $$|A+A| \gg |A| \log |A| / \log \log |A|.$$ Konyagin \cite{Ko} and Garaev \cite{Ga} showed if $A$ is a convex set then $$|A \pm A| \gg |A|^{3/2}.$$  Schoen and Shkredov improved this to $$|A-A| \gg |A|^{8/5} \log^{-2/5} |A| , \ \ \ \ |A+A| \gg |A|^{14/9} \log^{-2/3} |A|,$$ which is the current state of the art. It is conjectured $$|A \pm A| \gg_{\epsilon} |A|^{2 - \epsilon}.$$
Elekes, Nathanson, and Ruzsa \cite{El} showed that for any convex set $A$ and any $B$, \begin{align}\label{convexAB}
	|A+B| \gg |A| |B|^{1/2}.
\end{align} Finally Solymosi \cite[Theorem 1.1]{So} generalized \eqref{convexAB} and showed that if the differences $a_{i+1} - a_i$ are distinct for $1\leq i \leq k-1$, then \begin{align}\label{ThmSo}
|A+B| \gg |A||B|^{1/2},
\end{align} and a construction in the same paper, due to Ruzsa, showed this bound is sharp.
We generalize this result of Solymosi \cite[Theorem 1.1]{So}. 

Fix $r \geq 1$ an integer. We say a set $A$ has {\em distinct consecutive $r$--differences} if for $1 \leq i \leq k -r$, $$(a_{i+1} - a_i , \ldots , a_{i+r} - a_{i + r -1})$$ are distinct. 

\begin{theorem}\label{main2}
	Let $A$ and $B$ be finite subsets of real numbers and suppose $A$ has distinct consecutive $r$--differences. Then $$|A+ B| \gg e^{-r(\log 2 + 1)} |A||B| ^{1/(r+1)}.$$ The implied constant is absolute. Also, there exist sets such that the above inequality is sharp up to the constant. 
\end{theorem}
The case $r=1$ is in \cite[Theorem 1.1]{So}. Our Theorem \ref{main2} applies to more general sets than addressed in \cite{So} but our bound is smaller by a power of the size of $B$ when $r > 1$. We also show  that Theorem \ref{main2} is best possible, up to the constant, utilizing ideas from the construction of de Bruijn sequences. 

Here we study only the non-symmetric version of finding lower bounds for $|A+B|$ where $A$ has distinct consecutive $r$--differences. We expect improvements to Theorem~\ref{main2} in the case $B = A$. 

\begin{ques}\label{ques1}
	What is the largest $\theta_r$ such that for every $A \subset \mathbb{Z}$ with distinct consecutive $r$--differences, one has $$|A+A| \gg_{r} |A|^{1 + \theta_r/(r+1)}.$$
\end{ques}

Theorem \ref{main2}, with $B = A$, asserts that $\theta_r \geq 1$, while we provide a construction below that shows $\theta_r \leq 2$. We remind the reader that any convex set has distinct consecutive $1$--differences. So Question \ref{ques1} generalizes the aforementioned question of Erd\"{o}s regarding convex sets. 

We remark that the notion of distinct consecutive $r$--differences extends naturally to $\mathbb{F}_p$. Here we order the elements of $A$ in accordance of their smallest positive representation in the integers and adopt the same definition for distinct consecutive $r$--differences as above. Then the proof of Theorem~\ref{main2} also proves the following. 

\begin{theorem} Let $A , B \subset \mathbb{F}_p$, and suppose that $A$ has distinct consecutive $r$--differences. Then either $A+B = \mathbb{F}_p$ or $$|A+ B| \gg e^{-r(\log 2 + 1)} |A||B| ^{1/(r+1)}.$$
\end{theorem}

There is a generalization of Theorem 3 in \cite{So} for distinct consecutive $r$--differences, which requires the following definition. Let $A_1 , \ldots , A_d$ be nonempty finite subsets of real numbers all of cardinality $k$. We say that $A_1 , \ldots , A_d$ have {\em distinct $d$-tuples of consecutive $r$--differences} if the $(dr)$-tuples,
$$(a_{1,i+1} - a_{1,i} , \ldots , a_{1,i+r} - a_{1,i + r - 1} , \ldots , a_{d,i+1} - a_{d,i} , \ldots , a_{d,i+r} - a_{d,i + r - 1})$$ are distinct for $1 \leq i \leq k -r$. We have the following generalization of Theorem \ref{main2}. 

\begin{theorem}\label{multidim} Suppose $A_1 , \ldots , A_d \subset \mathbb{R}$ have finite size $k \geq 1$ and have distinct $d$-tuples of consecutive $r$--differences. Let $B_1 , \ldots B_d \subset \mathbb{R}$ be nonempty finite sets of real numbers of cardinality $\ell_1 , \ldots , \ell_d$. Then $$|A_1 + B_1| \cdots |A_d + B_d| \gg_{\beta , d} (k^{d r+1} \ell_1 \cdots \ell_d)^{1/(d(r+1))}.$$
\end{theorem}

The proof of Theorem \ref{main2} can be used to obtain an upper bound for the size of distinct $r$--differences of the set $A$ (Proposition~\ref{dcd}). This upper bound is not sharp when the set $A$ has some additive structure. In particular, let $\alpha$ be a real irrational number and we consider the set of points 
\begin{align*}
	S_\alpha(N):=\{\{n\alpha\}: 1\leq n\leq N\} = \{a_1 < \ldots <  a_N\}\subset \mathbb{R}/\mathbb{Z}. 
\end{align*}
Here we identify $\mathbb{R} / \mathbb{Z}$ with $[0,1)$ and then use the natural ordering on $[0,1)$. 
Since $|A+A| \ll |A|$, the above theory suggests that $A$ has few distinct consecutive $r$--differences. In fact,
in 1957 Steinhaus conjectured that there are at most $3$ distinct consecutive $1$--differences in $S_\alpha(N)$. This was proved by S\'os \cite{Sos1,Sos2} as well as \'Swierczkowski \cite{Swier}. Now we consider the set of distinct consecutive $r$--differences in $S_\alpha(N)$ defined via
$$D_r(S_{\alpha}(N)) := \{ (a_{i+1} - a_i , \ldots , a_{i +r} - a_{i+ r -1}) : a_i\in S_{\alpha}(N)\},$$ 
where $a_{i+N}=a_i$. Since there are at most 3 distinct $1$--differences in $S_\alpha(N)$, there are at most $3^r$ distinct consecutive $r$--differences in $S_\alpha(N)$. However, we prove that the size of $ D_r(S_\alpha(N))$ is much smaller than $3^r$ due to the structure of $S_\alpha(N)$. 

\begin{theorem}\label{2r+1diff}
	There are at most $2r+1$ distinct consecutive $r$--differences in $S_\alpha(N)$. 
\end{theorem}

We also consider a dual problem studied by Slater in \cite{slater}. Given $\phi , \theta \in (0,1)$, let the set of returning times be $$R_\theta(\phi):=\{T\in \NN^+ : \{T\theta\}<\phi\} = \{T_1 < T_2 < \ldots \}.$$
In \cite{slater,slater2}, Slater proved that there are at most $3$ distinct consecutive $1$--differences in $R_\theta(\phi)$. We generalize this result to consecutive $r$--differences. 
\begin{theorem}{\label{returntimethm}}
	There are at most $2r+1$ distinct consecutive $r$--differences in $R_\theta(\phi)$.
\end{theorem}                                     
\section{Distinct consecutive $r$--differences}
We first discuss the construction that shows Theorem \ref{main2} is best possible up to the constant. To do this, we utilize a lemma from graph theory that generalizes a construction of Ruzsa  presented in \cite{So}. 

\begin{lemma}\label{sequence} Let $S$ be any set. There exists a sequence $s_1 , \ldots , s_k$ of elements of $S$ (with repeats) such that 
	\begin{enumerate}
		\item [(a)]The ordered $(r+1)$-tuples $(s_{j} , \ldots , s_{j + r})$ are distinct for $1 \leq j \leq k$, where $s_{j+k} = s_j$,\label{one}
		\item [(b)]$k = |S| (|S| - 1)^{r}$,\label{two}
		\item [(c)]$s_j \neq s_{j+1}$, \text{for} $1 \leq j \leq k$.\label{three} 
	\end{enumerate}
\end{lemma}

We remark that if the last condition were eliminated and $k$ were replaced by $|S|^{r+1}$, then we would be in search of a de Bruijn sequence. These are known to exists and are well-studied. Indeed we modify a construction of de Bruijn sequences in the proof. 

\begin{proof}[Proof of Lemma \ref{sequence}]
	We define a directed graph $(V,E)$. We define $V$ to be all of the $ |S|(|S| - 1)^{r - 1}$ ordered tuples of size $r$ with elements from $S$ such that no two consecutive elements are the same. To define $E$, we say $x \to y$ if the last $r -1 $ elements of $x$ are the same (and in the same order) as the first $r - 1$ elements of $y$. Then the outdegree and indegree of any vertex is $|S| - 1$, and it is easy to see that $(V,E)$ is strongly connected. By a standard result in graph theory, there exists an Eulerian circuit in $(V,E)$, say $v_1 , \ldots , v_k$. Setting $s_j$ to be the first coordinate of $v_j$ for $1 \leq j \leq k$ gives the claim. 
\end{proof}

\begin{proof}[Proof of sharpness in Theorem \ref{main2}]
Now we are ready to show that Theorem \ref{main2} is sharp up to a constant. Let $S$ be any finite integer Sidon set and $s_1 , \ldots , s_k$ be the sequence of elements of $S$ as given by Lemma \ref{sequence}. We define sets $A , B \subset \mathbb{Z}^2$ via $$A := \{(i , s_i) : 1 \leq i \leq k\} , \ \ \ \ B := \{ (i , 0) : 1 \leq i \leq k\}.$$ Since $S$ is a Sidon set and by part (c) of Lemma \ref{sequence},
$$((i+1, s_{i+1}) - (i,s_i) , \ldots , (i+r, s_{i+r}) - (i+r - 1,s_{i+r-1})),$$ uniquely determines $$(s_i , \ldots , s_{i+r}).$$

By part (b) of Lemma \ref{sequence}, $(s_i , \ldots , s_{i+r})$ are distinct for $1 \leq i \leq k - r$. To achieve subsets of $\mathbb{Z}$ rather than $\mathbb{Z}^2$, we use the standard trick to define an injection $\phi : \mathbb{Z}^2 \to \mathbb{Z}$ via $$\phi(u,v) = Mu + v,$$ for an $M > 2 (\max S - \min S)$ chosen sufficiently large so that $|\phi(A)+\phi (B)| = |A+B|$. 
Thus $\phi(A)$ has the property of distinct consecutive $r$--differences. But $$|\phi(A) + \phi(B)| = |A + B| \leq 2k |S| \ll  |A||B|^{1/(r+1)} .$$ 
\end{proof}
We remark the set $\phi(A)$ as defined above is an example that shows $\theta_r \leq 2$ in Question \ref{ques1}. That is, we have $$|A+A| \ll |A|^{1 + 2/(r +1)}.$$ This follows from the additive version of Ruzsa's triangle inequality, which asserts $$|A+A||B| \leq |A+B|^2 \ll |A|^{2 + 2/(r +1)}.$$ Alternatively, one could compute $|A+A|$ explicitly to see that $|A|^{1 + 2/(r +1)}$ is the right order of magnitude of $|A+A|$.

We now move onto the proof of the inequality in Theorem~\ref{main2}, which  can be derived as a corollary in a more general setting. Given any set $A$ of size $k$, we let 

$$D_{r}(A) = \{(a_{i+1} - a_i , \ldots , a_{i+r} - a_{i + r -1}) : 1 \leq i \leq k - r\}.$$

\begin{prop}\label{dcd} Let $B$ be any set of size $\ell$ and $A$ as above. Then $$|A+B| \gg e^{-r (\log 2 + 1)} D_{r}(A) |B|^{1 / (r +1)}.$$
\end{prop}
\begin{proof}[Proof of Theorem~\ref{main2}]
The first part of Theorem \ref{main2} follows immediately from Proposition \ref{dcd} by observing that if $A$ has the property of distinct consecutive $r$--differences, then $|D_{r}(A)| = k - r$. 
\end{proof}

\begin{proof} [Proof of Proposition \ref{dcd}]
If $|D_{r}(A)| \leq 2 r$, Proposition \ref{dcd} follows from $|A+B| \geq |B|$, so we suppose $|D_r(A)| > 2r$. 

For each $d \in D_{r}(A)$, we choose an $1 \leq i(d) \leq k - r$ so that $$d = (a_{i(d) + 1} - a_{i(d)} , \ldots , a_{i(d) + r} - a_{i(d) + r - 1}).$$ 

Denote $$J_A:=\{ i(d): d\in D_r(A) \}.$$

Let $C = A+B$ and partition $$C = C_1 \cup \ldots \cup C_t,$$ such that for $u < v$ every element of $C_u$ is less than every element of $C_v$. 

Let $$X_u = \{(x_1 , \ldots , x_{r+1} )\in (A+B)^{r+1} : x_1 , \ldots , x_{r+1} \in C_u \ {\rm are \ distinct}\}.$$

The proof relies on double counting
\begin{equation}\label{dub} |X| , \ \ \ X = \bigcup_{u=1}^t X_u.\end{equation}

For each $1 \leq u \leq t$, we have that $C_u$ contains at most ${|C_u| \choose r+1}$ subsets of size $r + 1$. Thus have the upper bound for \eqref{dub} $$|X| \leq \sum_{u=1}^t {|C_u| \choose r + 1}.$$

Now we move on to a lower bound for \eqref{dub}. The key observation is that the  $(r+1)$--tuples \begin{equation}\label{tuple} (a_i + b , \ldots , a_{i+r} + b) , \ \ i \in J_A , \ \ b \in B,\end{equation} are distinct. Indeed given a $(r+1)$-tuple in \eqref{tuple} , we may recover $$(a_{i + 1} - a_{i} , \ldots , a_{i + r} - a_{i + r - 1}),$$ which determines $a_i , \ldots , a_{i+r}$ by our definition of $J_A$. Now an $(r+1)$--tuple in \eqref{tuple} is in $X$, unless the elements do not lie in the same $X_u$. For a fixed $b$, at most $(t-1)r$ of the $(r+1)$--tuples do not lie in the same $X_u$. Allowing $b$ to vary, and noting there are $D_r(A)|B|$ elements of \eqref{tuple}, we find $$|X| \geq (D_{r}(A) - (t-1) r) |B|.$$

Putting the upper and lower bounds for \eqref{dub} together, we have $$(D_{r}(A) - (t-1) r) |B| \leq \sum_{u=1}^t {|C_u| \choose r + 1}.$$ We choose $t = \lfloor D_r(A)/(2 r) \rfloor$ (which by assumption is at least 1) and $C_1 , \ldots , C_t$ to differ in size by at most 1, which implies $||C_u| - |C|/t| \leq 1.$ Proposition \ref{dcd} follows from Stirling's formula and a straightforward calculation. 
\end{proof}

We now give an informal sketch of a proof of Theorem \ref{multidim} below, which is similar to Theorem \ref{main2}. We also refer the reader to the proof of Theorem 3 in \cite{So}.

\begin{proof}[Sketch of proof of Theorem \ref{multidim}]

The case $k < 2rd$ is trivial, so we assume $k \geq 2rd$. For $1 \leq m \leq d$, let $A_m = \{a_{m,1} , \ldots , a_{m,k}\}$, $B_m = \{b_{m,1} , \ldots , b_{m, \ell_m}\}$ and $C_m = A_m + B_m$. Partition $C_m = C_{m,1} \cup \ldots \cup C_{m,t_m}$ as in Proposition \ref{dcd}. Double count the number of $$(a_{1 i} + b_{1j} , \ldots , a_{1 , i+r} + b_{1j} , \ldots , a_{d ,i} + b_{d,j} , \ldots , a_{d , i+r} + b_{d,j}    ),$$ such that $a_{m ,i} + b_{m,j} , \ldots , a_{m, i+r} + b_{m,j}$ all lie in a single $C_{m,u}$. Similar to Theorem 3 in \cite{So}, this implies an inequality of the form 
$$(k - r \sum_{m=1}^d (t_m-1))\ell_1 \cdots \ell_d \leq \sum_{u_1 = 1}^{|C_1 |} \cdots \sum_{u_d= 1}^{|C_d|} {|C_{1, u_1}| \choose r+1} \ldots {|C_{d, u_d}| \choose r+1}.$$ Choosing $t_m = \lfloor k/(2rd) \rfloor$ and the $C_{m,j}$ to differ in size by at most 1 implies Theorem~\ref{multidim}. 
\end{proof}

\section{Distinct consecutive $r$--differences of $\{n\alpha\}$}\label{nalpha}

\begin{proof}[Proof of Theorem \ref{2r+1diff}]
	Recall from the introduction that $$S_\alpha(N):=\{\{n\alpha\}: 1\leq n\leq N\} = \{x_1 < \ldots <  x_N\}\subset \mathbb{R}/\mathbb{Z}, $$ and $$D_r(S_{\alpha}(N)) := \{ (x_{i+1} - x_i , \ldots , x_{i +r} - x_{i+ r -1}) : x_i\in S_{\alpha}(N)\}.$$ To obtain an upper bound for $\#D_r(S_{\alpha}(N))$. We consider the set 
	 \begin{align*}
D_r(\alpha,N):=\{(\{a_{i+1}\alpha\}-\{a_i\alpha\},\cdots,\{a_{i+r}\alpha\}-\{a_{i+r-1}\alpha\} ): \\ \{(a_{i}-1)\alpha\},\cdots,\{(a_{i+r}-1)\alpha\}  \\\text{ are not consecutive elements in } S_\alpha(N)\}, 
	\end{align*} which contains $D_r(S_{\alpha}(N))$. Thus to prove Theorem \ref{2r+1diff}, it is enough to give an upper bound on $\#D_r(\alpha,N)$. The case when  $\{a_{i}\alpha\},\cdots,\{a_{i+r}\alpha\} $ are consecutive elements in $S_\alpha(N)$ while $\{(a_{i}-1)\alpha\},\cdots,\{(a_{i+r}-1)\alpha\} $ are not consecutive elements in $S_\alpha(N)$ can only happen if

\begin{enumerate}
	\item $a_j-1=0$ for some $i\leq j\leq i+r.$
	\item there exists $a_k$ such that $\{a_k\alpha\}$ is between $\{(a_j-1)\alpha\}$ and $\{(a_{j-1}-1)\alpha\}$ for some $i+1\leq j\leq i+r.$
\end{enumerate}
The first case happens if and only if $a_j=1$ for some $i\leq j\leq i+r$. The second case happens if and only if $a_k=N$ for some $i+1\leq k\leq i+r$. Thus there are at most $2r+1$ distinct consecutive $r$--differences in the sequence $S_\alpha(N)$. 
\end{proof}
\noindent Next we give a description of the pattern of the consecutive $r$--differences in $S_\alpha(N)$.

\begin{lemma}\label{1Npattern}
Suppose $\{n_1\alpha\},\{n_2\alpha\},\cdots,\{n_k\alpha\}$ are consecutive elements in $S_\alpha(N)$. Then $\{(N+1-n_k)\alpha\},\cdots,\{(N+1-n_2)\alpha\} ,\{(N+1-n_1)\alpha\} $ are consecutive elements in $S_\alpha(N)$.
\end{lemma}

\begin{proof}
The map $\{j\alpha\} \mapsto \{(N+1-j) \alpha\}$ is a permutation of $S_\alpha(N)$. Since $\{m\alpha\}=1-\{-m\alpha\}$ and $\{n_1\alpha\}<\{n_2\alpha\}< \cdots<\{n_k\alpha\}$, it follows that $\{(N+1-n_1)\alpha\}<\{(N+1-n_2)\alpha\}<\cdots<\{(N+1-n_k)\alpha\}$. There cannot be an $m$ such that $\{m\alpha\}$ is between $\{(N+1-n_i)\alpha\}<\{(N+1-n_j)\alpha\}$, since it would follow that  $\{(N+1-m)\alpha\}$ is in between $\{n_j\alpha\}$ and $\{n_i\alpha\}$, a contradiction.
\end{proof}
\begin{corollary}
	Suppose $L_1\alpha,\cdots,L_t\alpha, \alpha, R_1\alpha,\cdots,R_k\alpha\subset\mathbb{R}/\mathbb{Z}$ are the consecutive terms around $\{\alpha\}$ in $S_\alpha(N)$. Then $(N+1-R_k)\alpha,\cdots, (N+1-R_1)\alpha, N\alpha, ((N+1-l _t)\alpha),\cdots,(N+1-L_1)\alpha\subset\mathbb{R}/\mathbb{Z}$ are consecutive terms around $\{N\alpha\}$.
\end{corollary}
\begin{theorem}\label{rdiffpattern}
	Suppose $\alpha$ is irrational and $N$ is large enough so that there the $2r +1$ elements around $\alpha$ in $\mathbb{R}/\mathbb{Z}$ are all in $[0,1)$
	$$L_1\alpha,\cdots,L_r\alpha, \alpha, R_1\alpha,\cdots,R_r\alpha,$$ Let $$1=p_0<\cdots <p_i<p_{i+1}<\cdots< p_{2r},$$ be a reordering of the set $$\{1,L_1, L_2,\cdots,L_r, N+2-R_1,\cdots,N+2-R_r\}.$$ 
	Then $2r+1$ consecutive $r$--differences in $S_\alpha(N)$ are given by  $$d_r(\{p_i\alpha\}),
	 \ i=0,1,\cdots,2r,$$ where $d_r(x)$ denote the consecutive $r$--difference starting from $x$ in $S_\alpha(N)$ and $$d_r(\{n\alpha\})=d_r(\{p_i\alpha\}), \text{ for }p_i\leq n < p_{i+1}.$$ 
\end{theorem}
\begin{proof}
	The $2r+1$ consecutive differences are determined by the sequence 	$$L_1\alpha,\cdots,L_r\alpha, \alpha, R_1\alpha,\cdots,R_r\alpha.$$ For $r+1$ of them, the consecutive $r$--differences are given by $r+1$ consecutive numbers in the list. Thus $L_1,L_2,\cdots,1$ determines the $r+1$ consecutive $r$--differences in $S_\alpha(N)$, which are given by $d_r(\{L_t\alpha\})$ for $t=1,\cdots,r$ and $d_r(\{\alpha\})$. The remaining $r$ of the consecutive $r$--differences in $S_\alpha(N)$ are determined by $r+1$ consecutive numbers around $N\alpha$. From Lemma \ref{1Npattern}, the $r$ neighbours around $N\alpha$ in $\mathbb{R}/\mathbb{Z}$ are $$(N+1-R_r)\alpha,\cdots,(N+1-R_1)\alpha, N\alpha, (N+1-L_r)\alpha,\cdots,(N+1-L_1)\alpha.$$ Thus each consecutive $r$--difference is given by $r+1$ of the consecutive numbers in $$(N+1-R_r)\alpha,\cdots,(N+1-R_1)\alpha, (N+1-L_r)\alpha,\cdots,(N+1-L_1)\alpha,$$ which is determined by $(N+1-R_r)\alpha,\cdots,(N+1-R_1)\alpha$. In fact, they are given by $d_r(\{(N+2-R_l)\alpha\})$, where $l=1,\cdots,r$. In summary,  $$D_r(S_\alpha(N))=\{d_r(\{\alpha\}),d_r(\{L_1\alpha\}),\cdots,d_r(\{L_r\alpha\}), d_r(\{N+2-R_1\alpha\}),\cdots,d_r(\{N+2-R_r\alpha\})\}$$ gives the $2r+1$ consecutive $r$--differences in $S_\alpha(N)$, and 
	$$d_r(\{n\alpha\})=d_r(\{(n+m)\alpha\}),$$ as long as $n+m\leq N$ and $n+m$ doesn't belong to $$\{1,L_1,\cdots,L_r,N+2-R_1\cdots,N+2-R_r\}.$$
	 So for any $p_i\leq n<p_{i+1}$, we have $n-p_i\geq 0$ thus  $d_r(\{n\alpha\})=d_r(\{p_i\alpha\})$.   
	\end{proof}
	\begin{example}
		Take $\alpha=\log_{10}2$, $r=3$, and $N=100$. The $r$ neighbours around $\alpha$ are $$74\alpha, 84\alpha, 94\alpha, \alpha, 11\alpha, 21\alpha, 31\alpha\subset\mathbb{R}/\mathbb{Z}.$$ Applying Theorem \ref{rdiffpattern}, $$\{1,71,74,81,84,91,94\}$$ determines the $7$ distinct consecutive $3$--differences for $S_{\log_{10}2}(100)$. And given any $1\leq n\leq 100$, $d_3(\{n\alpha\})$ can be found by determining which of the following intervals $n$ belongs to $$[1,70],[71,73],[74,80],[81,83],[84,90],[91,93],[94,100].$$ 
	\end{example}
	\begin{theorem}
		Let $$S_{\alpha,\lambda_1,\cdots,\lambda_k}(N_1,\cdots,N_k) := \{\{\alpha n_i+\lambda_i \}|1\leq n_i\leq N_i, i=1,\cdots,k\}.$$ There are at most $(2r+1)k$ distinct consecutive $r$--differences in $S_{\alpha,\lambda_1,\cdots,\lambda_k}(N_1,\cdots,N_k)$. 
	\end{theorem}

\begin{proof}
We sketch the proof which is similar to the case when $k=1$ as in Theorem \ref{2r+1diff}. Let $N = N_1 \cdots N_k$ and denote the set  $$S_{\alpha,\lambda_1,\cdots,\lambda_k}(N_1,\cdots,N_k):= \{a_1 < \ldots < a_N\}.$$ Then the distinct consecutive $r$--differences can be represented by the $(r+1)$-tuple $(a_i,a_{i+1},\cdots,a_{i+r})$ such that $$a_{i}-\alpha , \cdots , a_{i +r} - \alpha$$ are not consecutive elements in $S_{\alpha,\lambda_1,\cdots,\lambda_k}(N_1,\cdots,N_k)$. This can only happen if one of the coordinates of the tuple $(a_i,a_{i+1},\cdots,a_{i+r+1})$ is of the form $\alpha+\lambda_j$ for some $j$, or there is a point of the form $N_j\alpha+\lambda_j$ between $a_{i}$ and $a_{i+1}$. This gives at most $2r+1$ $r$-tuples  $(a_i,a_{i+1},\cdots,a_{i+r})$ for each $j$. 
\end{proof}

\begin{theorem}\label{dcd2}
	Let $B$ be a finite subset of $\mathbb{R}/\mathbb{Z}$, then any subset $A$ of $B$ has at most $$C_r|B|^{1-\frac{1}{r+1}}\frac{|A+B|}{|B|}+r$$ distinct consecutive $r$--differences for some $C_r > 0$. One may choose $C_r=\frac{2r^{1-\frac{1}{r+1}}}{(r+1)!^{\frac{1}{r+1}}}$.
\end{theorem}

We omit the proof, as it is nearly identical to that of Proposition \ref{dcd}. We remark that Theorem \ref{dcd2} is a generalization of Theorem 1 in \cite{antal}. We now show that up to the constant, Theorem \ref{dcd2} is best possible. Let $S = \{1 , \ldots , |S|\}$. By Lemma \ref{sequence}, there exists $s_1 , \ldots , s_k$ such that 

\begin{itemize}
\item The ordered $r$-tuples $(s_{j} , \ldots , s_{j + r-1})$ are distinct for $1 \leq j \leq k$, where $s_{j+k} = s_j$,
\item $k = |S| (|S| - 1)^{r-1}$,
\item for $1 \leq j \leq k$, $s_j \neq s_{j+1}$. 
\end{itemize}

We define a set $A = \{a_1 < \ldots < a_k\}$ where $$a_i := \sum_{j=1}^i s_j.$$ Then $A$ has distinct consecutive $r$--differences. Note that $a_k \leq |S|^{r+1}$, so we let $B = \{0 , \ldots , N\}$ where $N = |S|^{r+1}$, so that $A \subset B$. Note that $$|A| \asymp |S|^{r},  \ \ \ \ \  |B| = |S|^{r+1},$$ so that $|A| \asymp |B|^{1 - 1/(r+1)}$. To make these subsets of $\mathbb{R}/\mathbb{Z}$, we consider the map $\phi : \mathbb{Z} \to \mathbb{R}/\mathbb{Z}$ via $\phi (x) = x \alpha$ for a sufficiently small $\alpha$.

\section{Distinct consecutive $r$--differences of returning times}

We recall that for $0<\phi , \theta<1$, we have the set of returning times $$R_\theta(\phi) =\{T : \{T\theta\}<\phi\} = \{T_1 < T_2 < \ldots \}.$$

\begin{proof}[Proof of Theorem \ref{returntimethm}]
	We prove this theorem by induction on $r$. Let $s \in R_{\theta}(\phi)$ and $d_r(s)\in\mathbb{Z}^r$ such that $s$ is followed by $$s+d_r(s)^{(1)},s+d_r(s)^{(1)}+d_r(s)^{(2)},\cdots, s+\sum_{l=1}^rd_r(s)^{(l)}$$ in $R_\theta(\phi)$, where $d_r(s)^{(l)}$ denotes the $l^{th}$ coordinate of $d_r(s)$. When $r=1$, the problem was studied by Slater in \cite{slater}. Let $a,b$ be the least positive integers such that $$\alpha:=\{a \theta\}<\phi, \ \beta:=1-\{b\theta\}<\phi.$$ Then from the definition of $a,b$, we have $\phi>\max(\alpha,\beta)$ and $\phi\leq \alpha+\beta$. There are three types of $d_1(s)$ given as below
	\begin{equation}\left\{
		\begin{array}{lll}
			d_1(s)=a, & 0\leq \{s\theta\}<\phi-\alpha\\
			d_1(s)=a+b, & \phi-\alpha\leq \{s\theta\}<\beta\\
			d_1(s)=b, & \beta\leq \{s\theta\}<\phi .\\
		\end{array}\right.
	\end{equation}
	
	This means there is a partition of $[0,\phi)$ into three intervals, each of which determines uniquely $d_1(s)$ depending where $\{s\theta\}$ lies in the interval $[0,\phi)$. Now suppose, by induction, there are at most $(2r-1)$ distinct consecutive $(r-1)$--differences in $R_\theta(\phi)$ which are determined by a partition of $[0,\phi)$ into $(2r-1)$ intervals. That is to say there are numbers $0<g_i<\phi,\ i=1,\cdots, 2r-2$, such that  $$0=g_0<g_1\leq \cdots\leq g_{2r-2}<g_{2r-1}=\phi$$ gives a partition of $[0,\phi)$ into at most $(2r-1)$ intervals. There is an one-to-one correspondence between $[g_i,g_{i+1})$ and a consecutive $(r-1)$--difference in $R_\theta(\phi)$ (note that if there are less than $2r-1$ intervals then we allow $g_i=g_{i+1}$
). Now we consider a consecutive $r$--difference in $R_\theta(\phi)$. Depending on whether $\{s\theta\}$ lies in $[0,\phi-\alpha),[\phi-\alpha,\beta)$ or $[\beta,\phi)$, $s$ is either followed by $s+a,s+a+b,s+b$ in $R_{\theta}(\phi)$, respectively. Thus $\{ (s+d_1(s))\theta \}$ is determined as below:
		\begin{equation}\left\{
		\begin{array}{lll}
		d_1(s)=a, &  \alpha\leq \{(s+a)\theta\}<\phi,\\
		d_1(s)=a+b,& \phi-\beta\leq\{(s+a+b)\theta\}<\alpha,\\
		d_1(s)=b, &  0\leq\{(s+b)\theta\}<\phi-\beta.\\
		\end{array}\right.
		\end{equation}
		  It follows that $\phi-\beta , \alpha , g_0 , \ldots , g_{2r-1}$ gives rise to a partition of $[0,\phi)$ into at most $(2r+1)$ intervals, each of which corresponds uniquely to a consecutive $r$--difference, depending on which one of these intervals $\{(s+d_1(s)) \theta\}$ lies. In fact, depending on which intervals of $[g_i,g_{i+1})$,  $[0,\phi-\beta)$ (repectively $[\phi-\beta,\alpha),[\alpha,\phi)$) intersect, the possible $r-1$ returning times following $(s,s+b)$ (respectively $(s,s+a+b), (s,s+a)$) will be uniquely determined.
	
	 To illustrate, we give the example of $d_2(s)$. 
	For $d_2(s)$, there are three possibilities depending on $\alpha,\beta$ and $\phi$. 
	\begin{align}&0\leq \phi-\alpha<\phi-\beta<\beta<\alpha<\phi:\nonumber\\\
		&\left\{
		\begin{array}{lll}
			d_2(s)=(a,b), &  \{s\theta\}\in[0,\phi-\alpha)\\
			d_2(s)=(a+b,a+b), & \{s\theta\}\in[\phi-\alpha,2\beta-\alpha)\\
			d_2(s)=(a+b,b), & \{s\theta\}\in[2\beta-\alpha,\beta)\\
			d_2(s)=(b,a), &\{s\theta\}\in[ \beta,\phi-\alpha+\beta)\\
			d_2(s)=(b,a+b), & \{s\theta\}\in[\phi-\alpha+\beta,\phi)\\
		\end{array}\right.
		\\
	  &0\leq \phi-\beta<\phi-\alpha<\alpha<\beta<\phi:\nonumber\\\
		&\left\{
		\begin{array}{lll}
			d_2(s)=(a,a+b), & \{s\theta\}\in[0,\beta-\alpha)\\
			d_2(s)=(a,b), & \{s\theta\}\in[\beta-\alpha,\phi-\alpha)\\
			d_2(s)=(a+b,a), & \{s\theta\}\in[\phi-\alpha,\phi-2\alpha+\beta)\\
			d_2(s)=(a+b,a+b), & \{s\theta\}\in[\phi-2\alpha+\beta,\beta)\\
			d_2(s)=(b,a), & \{s\theta\}\in[\beta,\phi)\\
		\end{array}\right.
		\end{align}
		\begin{align}
		&0\leq \phi-\beta<\alpha<\phi-\alpha<\beta<\phi:\nonumber\\\
		&\left\{
		\begin{array}{lll}
			d_2(s)=(a,a), & \{s\theta\}\in[0,\phi-2\alpha)\\
			d_2(s)=(a,a+b), & \{s\theta\}\in[\phi-2\alpha,\beta-\alpha)\\
			d_2(s)=(a,b), & \{s\theta\}\in[\beta-\alpha,\phi-\alpha)\\
			d_2(s)=(a+b,a), & \{s\theta\}\in[\phi-\alpha,\beta)\\
			d_2(s)=(b,a), & \{s\theta\}\in[\beta,\phi)\\
		\end{array}\right.
	\end{align}
\end{proof}
	For rational $\theta$ there is a relation between the consecutive $r$--differences in $R_\theta(\phi)$ and $S_\theta(N)$, which can be found in \cite{slater}. Suppose $\theta=\frac{p}{q}$. Let $\alpha=\frac{p'}{q}$, where $pp'\equiv 1\pmod q$. Then we have  $$\{1\leq s\leq q: \{s\theta \}< \frac{N}{q}\}=q\cdot\{\{s'\alpha\}:1\leq s'\leq N\},$$ by mapping $s$ to $s\equiv sp'\pmod q$. Thus the consecutive $r$--differences of the set 
	$$\{n\leq q|\  \{s\theta\} <\frac{N}{q}\}$$ are $q$ times the consecutive $r$--differences of the set $$\{  \{s\alpha\} , 1\leq s\leq N\}.$$ For general $\theta$ and $\phi$, more complications will appear depending on representation of $\phi$ in terms of convergents of continued fraction expansion of $\theta$.
\section*{Acknowledgments} 
The authors would like to thank Antal Balog, Oliver Roche-Newton and Alexandru Zaharescu for useful discussions. We also thank Boris Bukh for useful suggestions to an earlier draft.

\end{document}